\documentclass[11pt, a4paper]{amsart}

\usepackage{a4wide}
\usepackage[all,cmtip]{xy}

\usepackage{numprint}

\usepackage[english]{babel}
\usepackage{scalerel, mathtools}
\usepackage[backref=page]{hyperref}
\usepackage{microtype}
\usepackage{fancyhdr}
\usepackage{amsmath, amssymb,color}
\usepackage{amsthm}

\usepackage[T1]{fontenc}
\usepackage{lmodern}\normalfont
\DeclareFontShape{T1}{lmr}{bx}{sc} { <-> ssub * cmr/bx/sc }{}

\newcommand{\Lg}{\mathfrak{g}}

\newcommand{\Lie}{\mathfrak{Lie}}

\newcommand{\Lt}{\mathfrak{t}}

\newcommand{\ham}{\mathfrak{ham}}

\newcommand{\norm}[1]{\Vert #1\Vert}

\newcommand{\delz}[1]{\frac{\partial}{\partial z_{#1}}}

\newcommand{\cp}{\bar{\partial}}

\newcommand{\e}{\mathrm{e}}

\newcommand{\Hom}{\mathrm{Hom}}

\newcommand{\Aut}{\mathrm{Aut}}

\newcommand{\aut}{\mathfrak{aut}}
\newcommand{\id}{\operatorname{\text{\sf id}}}
\renewcommand{\Re}{\operatorname{Re}}
\renewcommand{\Im}{\operatorname{Im}}

\newcommand{\E}{\mathcal{E}}

\newcommand{\Oo}{\mathcal{O}}
\newcommand{\LL}{\mathcal{L}}
\newcommand{\TT}{\mathbb{T}}

\newcommand{\al}{\alpha}
\newcommand{\be}{\beta}

\newcommand{\la}{\lambda}

\renewcommand{\phi}{\varphi}

\newcommand{\ce}{\mathcal{C}^\infty}
\newcommand{\CC}{\mathbb{C}}
\newcommand{\HH}{\mathbb{H}}
\newcommand{\RR}{\mathbb{R}}
\newcommand{\Ss}{\mathbb{S}}

\newcommand{\ZZ}{\mathbb{Z}}
\newcommand{\NN}{\mathbb{N}}

\newcommand{\Ka}{K\"{a}hler}

\newcounter{Mycounter}[section]
\newcounter{lemma}[section]
\newcounter{claim}[section]

\newcounter{fact}[section]

\newcounter{sublemma}[section]

\newcounter{corollary}[section]

\newcounter{theorem}[section]

\newcounter{conjecture}[section]

\newcounter{proposition}[section]

\newcounter{definition}[section]

\newcounter{example}[section]

\newcounter{remark}[section]

\newcounter{problem}[section]

\newcounter{question}[section]

 \stepcounter{thmx}

\makeatletter

\@addtoreset{equation}{section}

\@addtoreset{footnote}{section}

\makeatother

\renewcommand*{\backref}[1]{}
\renewcommand*{\backrefalt}[4]{%
    \ifcase #1 (Not cited.)%
    \or        (Cited on page~#2.)%
    \else      (Cited on pages~#2.)%
    \fi}
\usepackage{enumerate}
\def\eqref#1{(\ref{#1})}

\def\1{\sqrt{-1}\:}

\newcommand{\cntrct}                
{\hspace{2pt}\raisebox{1pt}{\text{$\lrcorner$}}\hspace{2pt}}

\makeatletter
\newcommand*\rel@kern[1]{\kern#1\dimexpr\macc@kerna}
\newcommand*\widebar[1]{%
  \begingroup
  \def\mathaccent##1##2{%
    \rel@kern{0.8}%
    \overline{\rel@kern{-0.8}\macc@nucleus\rel@kern{0.2}}%
    \rel@kern{-0.2}%
  }%
  \macc@depth\@ne
  \let\math@bgroup\@empty \let\math@egroup\macc@set@skewchar
  \mathsurround\z@ \frozen@everymath{\mathgroup\macc@group\relax}%
  \macc@set@skewchar\relax
  \let\mathaccentV\macc@nested@a
  \macc@nested@a\relax111{#1}%
  \endgroup
}
\makeatother

\makeatletter
\def\cleardoublepage{\clearpage\if@twoside \ifodd\c@page \else\hbox{}\thispagestyle{empty}\newpage
\if@twocolumn\hbox{}\newpage\fi\fi\fi} \makeatother

\title{Existence Criteria for LCK Metrics}
\author{Nicolina Istrati}
\address{Univ Paris Diderot, Sorbonne Paris Cit\'{e}, Institut de Math\'{e}matiques de Jussieu-Paris Rive Gauche, UMR 7586, CNRS, Sorbonne Universit\'{e}s, UPMC Univ Paris 06, F-75013, Paris, France}

\email{nicolina.istrati@imj-prg.fr}

\begin{document}
\begin{abstract} We investigate the relation between holomorphic torus actions on complex manifolds of LCK type and the existence of special LCK metrics. We show that if the group of biholomorphisms of such a manifold $(M,J)$ contains a non-real compact torus, then there exists a Vaisman metric on the manifold. Moreover, we show that if the group of biholomorphisms contains a compact torus whose dimension is half the real dimension of $M$, then $(M,J)$ admits an LCK metric with positive potential. As an application, we obtain a classification of manifolds of LCK type among all the manifolds having the structure of a holomorphic principal torus bundle. Moreover, we obtain new non-existence results for LCK metrics on certain products of complex manifolds. 
\end{abstract}
\maketitle

\section*{Introduction}

Locally conformally \Ka\ (LCK) metrics are natural conformal generalisations of \Ka\ metrics. Namely, a Hermitian metric on a complex manifold $(M,J)$ with fundamental form $\Omega$ is LCK if $d\Omega=\theta\wedge\Omega$ for some closed form $\theta$, called the Lee form. On the minimal cover of $M$ on which the pullback of $\theta$ becomes exact, given by $p:\hat M\rightarrow M$  with  $p^*\theta=d\phi$, $\phi\in\ce(\hat M,\RR)$, there exists a global \Ka\ metric $\Omega_K=\e^{-\phi}p^*\Omega$, and $(\hat M ,J, \Omega_K)$ is called the minimal \Ka\ cover of the LCK structure.

Any LCK metric on a manifold of \Ka\ type is globally conformal to a \Ka\ metric (\cite{v80}). For this reason, we will always assume tacitly that our manifolds are not of \Ka\ type, in order to study only strict LCK metrics. In this setting, a first obstruction appears for manifolds of LCK type, cf. \cite{v80} and \cite{g76}, namely: $0<b_1<2h^{0,1}$, where $h^{0,1}=\dim_\CC H^1(M,\Oo_M)$ and $b_1=\dim_\RR H^1(M,\RR)$. As a matter of fact, this is the only cohomological obstruction known for a general LCK manifold. Vaisman had conjectured that such a manifold should always have $b_{2k+1}$ odd for some $k\in\NN$, however this was disproved by the Oeljeklaus-Toma manifolds \cite{ot05}.

There are a few special LCK metrics which are better understood. The most important one is a \textit{Vaisman metric}, defined by the condition $\nabla^g\theta=0$, where $\nabla^g$ is the Levi-Civita connection determined by $g$. 
It can be seen that a Vaisman metric $(\Omega,\theta)$ on $(M,J)$ has the form 
\begin{equation}\label{potentialone}
\Omega=-adJ\theta+\theta\wedge aJ\theta, \ \ \ a\in\RR_{>0},
\end{equation} 
and the corresponding \Ka\ metric on $\hat M$ satisfies $\Omega_K=dd^c(a\e^{-\phi})$. Thus $\Omega_K$ has a positive potential. This was first noted by Verbitsky \cite{ve}, and as a consequence Ornea-Verbitsky \cite{ov10} introduced and started the study of the more general notion of a \textit{LCK metric with (positive) potential}. These are LCK metrics whose \Ka\ form on $\hat M$ satisfies 
\begin{equation}\label{potential}
\Omega_K=dd^c(p^*f\e^{-\phi}), \ \ \ f\in\ce(M,\RR). 
\end{equation} 
This class of metrics has the advantage of being closed under small deformations of the complex structure (\cite{ov10}, \cite{go14}), while the Vaisman manifolds are not (see \cite{b00}). Even more general than this is the notion of an \textit{exact LCK metric}, which is an LCK metric whose \Ka\ metric has the form:
\begin{equation}\label{exact}
\Omega_K=d(\e^{-\phi}\eta), \ \ \ \eta\in\E^1(M,\RR).
\end{equation}

The main objective of the present paper is to study the relation between the existence of special LCK metrics on a compact complex manifold and the group of biholomorphisms of the manifold. It turns out that this problem translates into certain properties of a compact torus acting homolorphically on the manifold. In order to give the precise statements, let us first specify these properties.

Let $(M,J)$ be a compact complex manifold and let $\TT\subset\Aut(M,J)$ be a compact torus of biholomorphisms of $M$, with Lie algebra $\Lt\subset\ce(TM)$. Moreover, let $[\theta]\in H^1(M,\RR)$ and let $\hat M_{[\theta]}$ be the minimal cover of $M$ on which $\theta$ becomes exact.
\begin{definition}\label{defTor}\begin{enumerate}
\item We say that the torus is \textit{purely real} if $\Lt\cap J\Lt=0$.
\item We say that $\TT$ is \textit{horizontal} with respect to $[\theta]$ if the action of $\TT$ on $M$ lifts to an action of $\TT$ on $\hat M_{[\theta]}$. Otherwise, we say that it is \textit{vertical} with respect to $[\theta]$.
\end{enumerate}
\end{definition}

We then have the following results:

\begin{theorem}\label{thVa}
Let $(M,J)$ be a compact complex manifold of LCK type, and let $\TT\subset\Aut(M,J)$ be a torus which is not purely real. Then $(M,J)$ admits a Vaisman metric.
\end{theorem}

This result generalizes a criterion of Y. Kamishima and L. Ornea in \cite{ko} for deciding whether a given LCK conformal class is Vaisman or not, in terms of the presence of a complex Lie group with certain properties. As a corollary of the proof of our result, we obtain an obstruction to the existence of a general LCK metric:

\begin{corollary}\label{notLCK}
Let $(M,J)$ be a compact complex manifold. If there exists a torus $\TT\subset\Aut(M,J)$ so that $\dim_\RR(\Lt\cap J\Lt)>2$ then $(M,J)$ admits no LCK metrics.
\end{corollary}

Moreover, we give an alternative proof (Section~\ref{ePotential}) of the following result of Ornea-Verbitsky, in which we construct explicitely a positive potential by means of an ODE:

\begin{theorem}(\cite{ov12} and \cite{ov17})\label{crPotential}
Let $(M,J)$ be a compact complex manifold, and let $\tau\in H^1(M,\RR)$ be the de Rham class of a Lee form of an LCK structure on $(M,J)$. If there exists $\Ss^1\subset\Aut(M,J)$ which is vertical with respect to $\tau$, then there exists $\theta\in\tau$ and $\Omega$ so that $(\Omega,\theta)$ is an LCK structure with positive potential.
\end{theorem}

As a corollary of these results and a previous one concerning toric LCK manifolds \cite{is17} we also obtain:

\begin{theorem}\label{maxTor}
Let $(M,J)$ be a compact complex manifold of complex dimension $n$. Suppose that the group of biholomorphisms of $(M,J)$ contains an $n$-dimensional compact torus $\TT^n$. Then, for any $\tau\in H^1(M,\RR)$ which is the class of a Lee form of an LCK metric, there exists an LCK metric with positive potential $(\Omega,\theta)$ so that $\theta\in\tau$. 
\end{theorem}

We should note here that the dimension hypothesis on the torus is necessary, as the Inoue-Bombieri surface $S^+$ (\cite{in}) admits an effective holomorphic action of $\Ss^1$, but no exact LCK metric (\cite{o16}).

We also tackle the following problem, related to the above results. Any LCK metric $(\Omega,\theta)$ defines two natural vector fields $B$ and $A=JB$, the Lee and anti-Lee vector fields, via: 
\begin{equation}\label{Leedef}
\iota_A\Omega=-\theta, \ \ \ \iota_B\Omega=J\theta.
\end{equation}
It is well known that, for a Vaisman metric, these vector fields are real holomorphic. It is natural to ask whether the converse holds, or under which conditions. In the recent paper \cite{mmo18}, Ornea-Moroianu-Moroianu find additional properties ensuring that an LCK metric with holomorphic Lee vector field is Vaisman, namely: if the metric has harmonic Lee form (i.e. is Gauduchon), or if $B$ has constant norm. Moreover, they construct an example of a non-Vaisman LCK metric with Lee vector field. In the present paper, we show:

\begin{proposition}
Let $(M,J,\theta,\Omega)$ be an LCK metric of the form \eqref{potentialone} with holomorphic Lee vector field. Then $\Omega$ is Vaisman. 
\end{proposition}

This criterion should be particularly useful when constructing examples, as it is easy to check. Moreover, we note that the example of \cite{mmo18} can be chosen with positive potential, so the above result is the sharpest statement one can get. Finally, let us mention that the example of non-Vaisman metric with holomorphic Lee field is constructed on a manifold of Vaisman type. Thus the question remains open whether a manifold admitting an LCK metric with holomorphic Lee field is of Vaisman type. 

As a direct application of \ref{thVa}, we obtain in Section~\ref{torusPrincipal} a classification of manifolds of LCK type among all the manifolds having the structure of a holomorphic torus principal bundle. This is analogous to the result of Blanchard \cite{bl} in the \Ka\ context.

In the last part, we discuss the issue of irreducibility in LCK geometry. From early time \cite{v80}, it was known that if one takes two compact LCK manifolds $(M_i,\Omega_i)$, $i=1,2$, the product metric is not LCK on $M_1\times M_2$. However, whether there might exist some other LCK metric on $M_1\times M_2$ has remained an open question, and in Section~\ref{anIrred}, we extend the known cases (\cite{t99}, \cite{opv}) in which this fails, as an application of \ref{thVa}. 


\textbf{Acknowledgments.} I am grateful to Andrei Moroianu for his encouragement and valuable suggestions that improved this paper.

\section{LCK metrics}

In this section, we fix a complex manifold $(M,J)$. A metric $g$ on $(M,J)$ is \textit{Hermitian} if $g(J\cdot,J\cdot)=g$. In this case, $g$ induces a fundamental form $\Omega=g(J\cdot,\cdot)$ of bidegree $(1,1)$ with respect to $J$. Conversely, a $(1,1)$-form $\Omega\in\E^{1,1}(M,\RR)$ is called positive and we write $\Omega>0$ if the symmetric tensor $\Omega(\cdot,J\cdot)=:g$ is positve definite, in which case $g$ is a Hermitian metric. This one-to-one correspondence will be used implicitly throughout the present text.

We begin with the equivalent definitions of a locally conformally \Ka\ (LCK) metric. Let $g$ be a Hermitian metric with fundamental form $\Omega$ on $(M,J)$.

\begin{definition} The metric $g$ is called LCK if one of the following equivalent facts holds: \begin{enumerate}[(a)]
\item There exists a real closed one-form $\theta$ on $M$, called \textit{the Lee form}, for which we have: 
\begin{equation}\label{LCK1}
d\Omega=\theta\wedge\Omega.
\end{equation}
\item $M$ is covered by open sets $\{U_\al\}_{\al\in I}$ so that for each $\al\in I$ there exists a \Ka\ metric $g_\al$ on $(U_\al, J)$ and a real function $\phi_\al\in\ce(U_\al,\RR)$ so that: 
\begin{equation}\label{LCK2}
g_{|_{U_\al}}=\e^{\phi_\al}g_\al.
\end{equation}
\item There exists a \Ka\ metric $\Omega_K$ on the universal cover with the induced complex structure $\pi:(\tilde M,J)\rightarrow (M,J)$  on which $\pi_1(M)$, seen as the deck group of $\pi$, acts by homotheties:
\begin{equation}\label{LCK3}
\gamma^*\Omega_K=\rho(\gamma)^{-1}\Omega_K, \ \ \ \gamma\in\pi_1(M), \ \ \rho(\gamma)\in\RR_{>0}.
\end{equation}
\end{enumerate}
\end{definition}

It is not difficult to see that indeed all the above conditions are equivalent, and for the details, one can consult the monograph \cite{do}. The Lee form is given on the open sets $U_\al$ by $\theta|_{U_\al}=d\phi_\al$. The pullback of the \Ka\ metrics $\{g_\al\}_{\al\in I}$ to $\tilde M$ glue up to a global \Ka\ metric, which corresponds precisely to $\Omega_K$. Moreover, it is easy to check that the constants given in \eqref{LCK3} form a group morphism $\rho:\pi_1(M)\mapsto (\RR_{>0},\cdot)$, $\gamma\mapsto\rho(\gamma)$. The kernel of $\rho$ is a normal subgroup of $\pi_1(M)$, so one can consider $\Gamma:=\pi_1(M)/\ker\rho$ and the corresponding Galois cover $p:\hat M\rightarrow M$ of deck group $\Gamma$. By definition, $\Omega_K$ is $\ker\rho$-invariant, so descends to a \Ka\ metric on $\hat M$. 

In fact, $\hat M$ is the minimal cover of $M$ on which the pullback of $\Omega$ is globally conformal to a \Ka\ metric. For this reason, we will call $(\hat M,\Omega_K)$ \textit{the minimal \Ka\ cover} corresponding to $\Omega$. Moreover, the pullback of $\theta$ becomes exact on $\hat M$ and one has: 
\begin{equation*}
p^*\Omega=\e^{\phi}\Omega_K, \ \ \ p^*\theta=d\phi, \ \ \ \phi\in\ce(\hat M,\RR).
\end{equation*}

Note that the notion of an LCK metric is conformal in nature. Thus, any relevant definition concerning a general LCK structure should be conformally invariant. We will denote by $[\Omega]=\{\e^f\Omega|f\in\ce(M,\RR)\}$ the conformal class of an LCK metric. Among the objects defined above, the de Rham class $[\theta]$, the morphism $\rho$ and the half-line of \Ka\ metrics $\RR_{>0}\Omega_K$ are indeed univoquely defined by the conformal class $[\Omega]$. 

For later use, let us introduce the set of de Rham classes of Lee forms of LCK structures:
\begin{equation*}
\LL(M,J):=\{[\theta]\in H^1(M,\RR)|\exists \Omega\in\E^{1,1}(M,\RR), \ \Omega>0, \ d_\theta\Omega=0\}.
\end{equation*}
We will say that $(M,J)$ is of LCK type if $\LL(M,J)$ is not empty.

In LCK geometry, an important role plays the differential operator: 
\begin{equation*}\label{dtheta}
\begin{split}
d_\theta:\E^k(M)\rightarrow \E^{k+1}(M)\\
d_\theta\eta=d\eta-\theta\wedge\eta.
\end{split}
\end{equation*}
This operator naturally appears when one is led to consider $\rho^{-1}$-equivariant forms on $\hat M$, such as the \Ka\ form. Indeed, equivariant forms are exactly pullbacks of forms from $M$ multiplied by $\e^{-\phi}$, and under this operation $d$ on $\hat M$ corresponds to $d_\theta$ on $M$, as for any smooth form $\al$ on $M$, one has the relation:
\begin{equation*}
d(\e^{-\phi}p^*\al)=\e^{-\phi}p^*(d_\theta\al).
\end{equation*}

In the same manner appears also the operator $d^c_\theta:=d^c-J\theta\wedge\cdot$. A simple, but very useful fact is the following lemma, cf. \cite{v85}:

\begin{lemma}\label{ldtheta}
Let $M$ be a connected differentiable manifold and $\theta$ a real-valued closed $1$-form on $M$. Then $d_\theta:\ce(M)\rightarrow\E^1(M)$ is injective if and only if $\theta$ is not exact. 
\end{lemma}

Most of the special LCK metrics defined in the introduction can also be given equivalent definitions in terms of the operator $d_\theta$. As such, an LCK structure $(\Omega,\theta)$ on $(M,J)$  is: 
\begin{enumerate}[(a)]
\item \textit{exact} if there exists a one-form $\eta$ on $M$ so that $\Omega=d_\theta\eta$. Note that, indeed, this is equivalent to \eqref{exact}. 
\item \textit{with potential} if there exists $f\in\ce(M,\RR)$ so that $\Omega=d_\theta d^c_\theta f$. This is equivalent to \eqref{potential}. Moreover, it is called \textit{with positive potential} if $f$ can be chosen positive. 
\end{enumerate}
Remark that these definitions are invariant by conformal transformations, so can also be used for conformal classes of metrics. 

\subsection*{Vaisman metrics} On the other hand, Vaisman metrics cannot be defined by the operator $d_\theta$ alone, but it is true that they admit constant potential. First of all, the Lee and anti-Lee vector fields $A$ and $B$ of a Vaisman structure $(g,\Omega,\theta)$ on $(M,J)$, defined by \eqref{Leedef}, have remarkable properties. The defining condition $\nabla^g\theta=0$ is also equivalent to $\nabla^gB=0$, as $B$ is the metric dual of $\theta$. This immediately implies that $B$ (and so also $A$) is of constant norm. Moreover, it is not difficult to see that $A$ and $B$ are real holomorphic and Killing. Finally, this also implies that $B$ is symplectic, which is equivalent to $\Omega$ admitting $f=\frac{1}{\norm B^2}\in\RR$ as a potential:
\begin{equation*}
\Omega=\frac{1}{\norm B^2}(-dJ\theta+\theta\wedge J\theta)=d_\theta d^c_\theta(\frac{1}{\norm B^2}).
\end{equation*}

Moreover, a Vaisman metric is a \textit{Gauduchon metric}, meaning that its Lee form is $d^*$-closed (and thus harmonic), where $d^*$ is the co-differential with respect to $g$. This is easy to see if one writes $d^\star=-\sum_{j=1}^{2n}\iota_{e_j}\nabla_{e_j}$, with $\{e_1, \ldots e_{2n}\}$ a local orthonormal real basis of $TM$. In particular, a Vaisman metric inherits the property of Gauduchon metrics of being unique in their conformal class up to multiplication by a positive constant, cf. \cite{g77}. For this reason, we will usually normalize a Vaisman metric to verify $\norm B=1$, which then implies that $\Omega=d_\theta d^c_\theta 1$.

\section{The Lee vector field}\label{theLee}

It is easy to see that, if the Lee vector field of an LCK metric is Killing, then the metric is Vaisman. Moreover, the same conclusion holds if the Lee vector field preserves the fundamental form, by a result of \cite{mm17}. However, it is not true that the holomorphicity of the Lee vector field implies the Vaisman condition. It was recently shown:

\begin{theorem}(\cite{mmo18})
Let $(M,J,g,\Omega,\theta)$ be a compact LCK manifold with holomorphic Lee vector field. If $B$ has constant norm, or if $g$ is Gauduchon, then $g$ is a Vaisman metric.
\end{theorem}

Moreover, we have the following simple result:
\begin{proposition}\label{potV}
Let $(M,J)$ be a compact manifold with an LCK metric with potential $1$: $\Omega=-dJ\theta+\theta\wedge J\theta$. If $\Omega$ has real holomorphic Lee vector field $B$, then $\Omega$ is Vaisman.
\end{proposition}

\begin{proof}
If $B$ is real holomorphic, then also $A=JB$ is. The Cartan formula and $\LL_A\theta=0$ imply:
\begin{align*}
0&=\LL_AJ\theta=d\iota_AJ\theta+\iota_AdJ\theta\\
=&-d(\theta(JA))+\iota_A(\theta\wedge J\theta-\Omega)\\
=&d(\norm{B}^2)-\theta\norm{B}^2+\theta\\
=&d_\theta(\norm{B}^2-1).
\end{align*}
Thus \ref{ldtheta} implies that $\norm{B}^2=1$. Using again Cartan's formula and the form of $\Omega$, we obtain:
\begin{align*}
\LL_B\Omega&=d\iota_B\Omega+\iota_B(\theta\wedge\Omega)\\
&=dJ\theta+\norm B^2\Omega-\theta\wedge J\theta\\
&=-\Omega+\Omega=0.
\end{align*}
Finally, since $B$ preserves both the complex structure and the symplectic form, it also preserves the metric. This implies that $\nabla^g\theta$ is antisymmetric. We therefore obtain: $0=d\theta=2\nabla^g\theta$, i.e. $g$ is Vaisman.
\end{proof}

In the paper \cite{mmo18}, the authors also construct an example of an LCK metric which is not Vaisman, but which has holomorphic Lee vector field, thus showing that one needs some additional hypotheses on $\Omega$ to ensure that it is Vaisman. We now present this example, with the remark that in the original construction, the metric can in fact be chosen with positive potential. This shows that the hypotheses in \ref{potV} cannot be relaxed. 

\begin{example}(\cite{mmo18})\label{Leeolo}
Let $(M,J,\Omega,\theta)$ be a compact Vaisman manifold with $\norm{\theta}^2=1$, and let $B$ be its Lee vector field. Suppose there exists a non-constant smooth function $f\in\ce(M,\RR)$ verifying $f>-1$ everywhere on $M$ and such that $df$ is colinear with $\theta$.  After taking the interior product with $B$, this last condition is more precisely $df=B(f)\theta$. Such functions exist any time $B$ generates an $\Ss^1$-action on $M$, for instance on the standard Hopf manifold.

Consider next the form: 
\begin{equation*}
\Omega':=\Omega+f\theta\wedge J\theta=d_{(1+f)\theta}(-dJ\theta).
\end{equation*}
As $f>-1$, $\Omega'$ is a strictly positive real $(1,1)$-form on $M$, and verifies $d\Omega'=(1+f)\theta\wedge\Omega'$. Thus $\Omega'$ is the fundamental form of an LCK metric with Lee form $\theta'=(1+f)\theta$. 
 
\begin{lemma}\label{notconf}
The Lee vector field of $\Omega'$ is $B$, and so also holomorphic. The metric $\Omega'$ is not conformal to any Vaisman metric.
\end{lemma}
\begin{proof}
As $\iota_B\Omega'=(1+f)J\theta=J\theta'$, $B$ is also the Lee vector field of $\Omega'$. Now suppose that there exists a Vaisman metric $\Omega''$ on $M$ so that $\Omega''=\e^h\Omega'$. By a theorem of K. Tsukada \cite{t97}, the Lee vector field of a Vaisman metric is unique on the manifold $M$ up to multiplication by a constant. Thus, we can suppose right from the beginning that the Lee vector field of $\Omega''$ is also $B$. Now this reads:
\begin{equation*}
\e^hJ\theta'=\e^h\iota_B\Omega'=\iota_B\Omega''=J\theta'+d^ch
\end{equation*}
that is: $dh+\theta'(1-\e^h)=0$, or also, after multiplying by $-\e^{-h}$: $d_{\theta'}(\e^{-h}-1)=0.$ As $\theta'$ has no zero, it is non-exact, so \ref{ldtheta} implies that $\e^{-h}=1$, i.e. $h=0$ and $\Omega'$ is Vaisman. But this last fact is impossible, as the norm of $B$ is non-constant: $\Omega'(B,JB)=\theta'(B)=1+f$.
\end{proof}

Suppose now that the flow of $B$ on $M$ is periodic: $\Phi^{2\pi}_B=\id_M$, and that we have a diffeomorphism $M\cong N\times \Ss^1$, where $N$ is a compact Sasaki manifold. Any non-constant function on $\Ss^1$, bounded bellow by $-1$, induces a function $f$ on $M$ verifying the desired properties.

\begin{lemma}\label{positiveP}
Under the above hypothesis, the metric $\Omega'$ admits a positive potential.
\end{lemma}
\begin{proof}
We think of $f$ as a function on $\RR$ which is $2\pi$-periodic, and we are looking for another positive function $g:\RR\rightarrow\RR$, also $2\pi$-periodic, verifying, when seen as a function on $M$:
\begin{equation}\label{ecu}
\Omega'=d_{\theta'}d^c_{\theta'}g.
\end{equation} 
The function $g$ we are looking for verifies that both $dg$ and $d\LL_Bg$ are colinear with $\theta$, which implies the following relations:
\begin{equation*}
dg=\LL_Bg\cdot\theta,\ \ \  d^cg=\LL_Bg\cdot J\theta, \ \ \  dd^cg=\LL_B^2g\cdot\theta\wedge J\theta+\LL_Bg\cdot dJ\theta.
\end{equation*}
With this in mind, \eqref{ecu} writes:
\begin{align*}
-dJ\theta+(1+f)\cdot\theta\wedge J\theta=&(\LL_Bg-g(1+f))dJ\theta+\\
+&(\LL_B^2g-\LL_Bf\cdot g-2(1+f)\LL_Bg+g(1+f)^2)\theta\wedge J\theta.
\end{align*}
Now, the two forms $-dJ\theta$ and $\theta\wedge J\theta$ are linearly independent, which implies that in the above equation, the corresponding coefficients preceding them must be equal. We denote by $t$ the variable on $\RR$, and identify $B$ with the vector field $\frac{d}{dt}$ on $\RR$. Seeing $f$ and $g$ as functions on $\RR$, \eqref{ecu} now becomes equivalent to:
\begin{align}
\frac{d}{dt}g-g(1+f)+1&=0\\
\frac{d^2}{dt^2}g-2(1+f)\frac{d}{dt}g-g\frac{d}{dt}f+g(1+f)^2-(1+f)&=0.
\end{align}
By differentiating the first equation, one obtains the second one, while the first ODE has a solution of the form:
\begin{equation*}
g(t)=(c-\int_0^t\e^{-F(s)}ds)\e^{F(t)}, \ \ \ \text{ with } F(t)=a+\int_0^t(f(s)+1)ds, \ \ \ a,c\in\RR.
\end{equation*}
Thus a solution $g$ of the above system exists, and now it is left for us to show that we can choose the constants $a$ and $c$ such that $g$ is moreover strictly positive and $2\pi$-periodic. 

Let us note that, because $f$ is $2\pi$-periodic, we have, for any $t\in\RR$:
\begin{equation*}
F(t+2\pi)=F(t)+b, \ \ \ \text{ where } b=\int_0^{2\pi}(f(s)+1)ds>0.
\end{equation*}
Thus we obtain:
\begin{align*}
g(t+2\pi)&=(c-\int_0^{2\pi}\e^{-F(s)}ds-\int_{2\pi}^{2\pi+t}\e^{-F(s)}ds)\e^{F(t)}\e^b\\
&=(c-K-\int_{0}^t\e^{-F(u)}\e^{-b}du)\e^{F(t)}\e^b\\
&=g(t)+\e^{F(t)}((c-K)\e^b-c)
\end{align*}
where $K=\int_{0}^{2\pi}\e^{-F(s)}>0$ and, for the second equality, we made the change of variable $s=u+2\pi$. Thus, in order for $g$ to be $2\pi$-periodic, we take $c:=\frac{K\e^b}{\e^b-1}>0$.
Finally, we need to see that $g$ is in fact positive, which is also equivalent to saying that $v(t):=c-\int_0^t\e^{-F(s)}ds$ is positive. Note that $\frac{d}{dt}v(t)=-\e^{-F(t)}<0$, so $v$ can change sign at most once, and the same is then true for the function $g$. On the other hand, $g$ is periodic and $g(0)=c\e^a>0$, thus $g$ is indeed everywhere positive.
\end{proof}
\end{example}

Note that, although the above example shows that there can exist non-Vaisman metrics with holomorphic Lee vector field, it is however constructed out of a Vaisman metric. So one can still ask the following question:

\begin{question}
Let $(M,J,\Omega)$ be a compact LCK manifold with holomorphic Lee vector field. Does there exist an LCK metric on $M$, not necessarily conformal to $\Omega$, which is Vaisman?
\end{question}

Also, recall that the Lee vector field of any Vaisman metric is uniquely determined up to multiplication by a positive constant, by \cite{t97}. A related question is then:  

\begin{question}
Suppose that the Lee vector field of an LCK metric on a manifold of Vaisman type is holomorphic. Is it then the Lee vector field of a Vaisman metric? 
\end{question}

\section{Existence of LCK metrics with positive potential} \label{ePotential}

Let us start by reviewing the notion of a \textit{vertical} action of a torus. For our discussion, it is enough to consider $\Ss^1$-actions. In what follows, we fix $M$ a compact smooth manifold and $\tau\in H^1(M,\RR)$ a de Rham class. By the universal coefficient theorem, we can also view $\tau\in\Hom(\pi_1(M),\RR)$. Then $\ker \tau$ is a normal subgroup of $\pi_1(M)$, so we can take $\hat M_\tau:=\tilde M/\ker\tau$, which is a normal cover of $M$. If $\theta\in\ce(T^*M)$ is a smooth representative of $\tau$, then $\hat M_\tau$ is the minimal cover of $M$ on which $\theta$ becomes exact.

Suppose $\Ss^1$ acts on $M$ with fundamental vector field $C$, and let $\Phi_t$ denote the corresponding $1$-periodic flow. By averaging $\theta$ to an $\Ss^1$-invariant form: $\theta':=\int_0^1\Phi_t^*\theta dt$, the de Rham class does not change, i.e. $[\theta']=[\theta]=\tau$, so we can just suppose that $\theta$ is $\Ss^1$-invariant. Now we have $0=\LL_C\theta=d(\theta(C))$, so $\theta(C)=a\in\RR$. Moreover, the value $a$ only depends on the de Rham class $\tau$: it is in fact $\tau$ evaluated on the homotopy class of an orbit of $\Ss^1$. 

We have the following simple characterisation of vertical actions:

\begin{lemma}\label{hor}
The action $\Ss^1$ is vertical for $\tau\in H^1(M,\RR)$ if and only if $\theta(C)\neq 0$ for some (and so any) $\Ss^1$-invariant representative $\theta\in\tau$.
\end{lemma}
\begin{proof}
In any case, $C$ lifts to a vector field, also denoted by $C$, to $\hat M_\tau$, generating an $\RR$-action on $\hat M$. Let us denote by $\hat \Phi_t$ the corresponding flow. We want to show the equivalence: $\hat\Phi_1=\id_{\hat M_\tau} \Leftrightarrow \int_\gamma \tau\neq 0$, where $\gamma=[\Ss^1.x]\in\pi_1(M,x)$ is the homotopy class of an $\Ss^1$-orbit through an arbitrary point $x\in M$. 

Denote by $\hat \pi:\hat M_\tau\rightarrow M$ the covering of deck group $\Gamma:=\pi_1(M,x)/\ker\tau\subset\Aut(\hat M_\tau)$, and let $p:\pi_1(M,x)\rightarrow \Gamma$ be the natural projection. Then $\int_\gamma\tau\neq 0$ if and only if $p\gamma\neq id_{\hat M_\tau}$. But $p\gamma=\hat\Phi_1$: indeed, for any $\hat x\in\hat\pi^{-1}(x)$, the curve $[0,1]\ni t\mapsto \hat\Phi_t(\hat x)$ is the unique lift from $\hat x$ of the loop $[0,1]\ni t\mapsto \Phi_t(x)$ representing $\gamma$. Thus the conclusion follows.
\end{proof}

Thus, given a de Rham class $\tau\in H^1(M,\RR)$, a torus action $\TT^n$ on $M$ lifts to an action of $\TT^n$ on $\hat M_\tau$ if and only if $\Lie(\TT^n)\subset\ker\theta$ for any smooth $\TT^n$-invariant closed one-form $\theta\in\tau$.

\begin{proof}[\textbf{Proof of \ref{crPotential}}]
We recall that, by hypotheses, we have a compact LCK manifold $(M,J,\Omega,\theta)$ and a vertical $\Ss^1$-action on $M$ with respect to $\tau=[\theta]$. Let $(\hat M,J,\omega)$ be the minimal \Ka\ cover of $(M,J,[\Omega])$. We denote by $D$ the real holomorphic vector field on $M$ generating the $\Ss^1$-action, as well as its lift to $\hat M$. By a standard average argument which does not change the de Rham class of $\theta$, we can suppose that both $\Omega$ and $\theta$ are preserved by $D$. In particular, $\LL_D\theta=0$ implies, by \ref{hor}, that $\theta(D)=\la\in\RR^*$, as the action is vertical. Let $C:=\frac{1}{\la}D$, so that $\theta(C)=1$.

Let $\theta=d\phi$ on $\hat M$, so that the \Ka\ form writes $\omega=\exp(-\phi)\Omega$. Then we have:
\begin{equation}\label{omega1}
\LL_C\omega=-\theta(C)\omega=-\omega.
\end{equation}  
Let us denote by $\eta$ the real one-form on $\hat M$ defined by $\iota_C\omega=\eta$. Then \eqref{omega1} together with Cartan's formula imply:
\begin{equation}\label{omega2}
\omega=-(d\iota_C+\iota_Cd)\omega=-d\eta.
\end{equation}
At the same time, using the fact that $\eta(JC)=\omega(C,JC)=\norm{C}^2_\omega:=f$, we have:
\begin{align*}
\LL_{JC}\eta=d\iota_{JC}\eta+\iota_{JC}d\eta=df-J\eta,
\end{align*} 
from which it follows:
\begin{align*}
\LL_{JC}\omega &=-d(df-J\eta)=dJ\eta\\
\LL_{JC}^2\omega&=dJ\LL_{JC}\eta=dd^cf+d\eta=dd^cf-\omega.
\end{align*}
If we let $\Phi_t$ denote the one-parameter group generated by $JC$ and denote by $\omega_t:=\Phi_t^*\omega$ and by $f_t=\Phi_t^*f$, the last equation reads:
\begin{equation}\label{omega3}
\frac{d^2}{dt^2}\omega_t=-\omega_t+dd^cf_t.
\end{equation}

Let now $g_t$ be the real-valued functions on $\hat M$ defined by the second order linear differential equation:
\begin{align}\label{omega4}
\frac{d^2}{dt^2}g_t+g_t=f_t, \ \ \ g_0=0, \ \ \ \frac{d}{dt}|_{t=0}g_t=0.
\end{align}

We want to show that $\omega_t=\cos t \omega+\sin t dJ\eta+dd^c g_t$. For this, consider the forms $\be_t:=\omega_t-(\cos t \omega+\sin t dJ\eta+dd^c g_t)$, $t\in\RR$. Using \eqref{omega3} and the definition \eqref{omega4} of the functions $g_t$, we have:
\begin{align*}
\frac{d^2}{dt^2}\be_t&=\frac{d^2}{dt^2}\omega_t+\cos t  \omega+\sin t dJ\eta-dd^c(\frac{d^2}{dt^2}g_t)\\
&=-\omega_t+dd^cf_t+\cos t \omega+\sin t dJ\eta-dd^c f_t+dd^c g_t\\
&=-\be_t.
\end{align*}
Thus, the forms $\be_t$ verify the following  homogeneous second order linear differential equation with initial conditions:
\begin{equation*}
\frac{d^2}{dt^2}\be_t+\be_t=0, \ \ \ \be_0=0, \ \ \ \frac{d}{dt}|_{t=0}\be_t=0.
\end{equation*}
By the uniqueness of the solution, we have then that for all $t\in\RR$, $\be_t$ vanishes identically, and so:
\begin{equation}\label{omega5}
\omega_t=\cos t \omega+\sin t dJ\eta+dd^c g_t, \ \ \ t\in\RR.
\end{equation}

Define now, using \eqref{omega5}, a new form $\omega'$ by: 
\begin{equation*}
\omega':=\frac{1}{2\pi}\int_0^{2\pi}\Phi_t^*\omega dt=dd^c\frac{1}{2\pi}\int_0^{2\pi}g_tdt
\end{equation*}
and let us denote by $g$ the function $1/2\pi\int_0^{2\pi}g_tdt$. As $\{\Phi_t\}_{t\in\RR}$ is a subgroup of biholomorphism of $\hat M$, $\omega'$ is a \Ka\ form on $\hat M$. 

Next, we want to show that the function $g$ is everywhere positive on $\hat M$. Note first that, as $\theta(C)=1$, $C$ has no zeroes so the function $f$ is everywhere positive. Moreover, as $JC$ is real holomorphic, it commutes with both $C$ and $JC$, so we have $f_t=\Phi_t^*(\omega(C,JC))=\omega_t(C,JC)$, which is also everywhere positive for any $t\in\RR$. Finally, let us fix $x\in \hat M$, and define the real-valued function on $\RR$  $h(t):=g_t(x)$. For any $t_0\in\RR$ where $h(t_0)=0$, equation \eqref{omega4} implies that $\frac{d^2}{dt^2}|_{t=t_0}h=f_{t_0}(x)>0$, so $t_0$ is a strict local minimum for $h$. It follows then that $h$ is everywhere semipositive on $\RR$, and strictly positive on a dense open set of $\RR$. Thus $g(x)=1/2\pi\int_0^{2\pi} h(t)dt>0$. 

Hence we can define $\theta':=d\ln g$. Note that by the uniqueness of the solution of \eqref{omega4}, the functions $g_t$ have the same $\Gamma$-equivariance as the functions $f_t$, or also as the function $f$. Here, $\Gamma$ denotes the deck group of the cover $\hat M\rightarrow M$. Also we should note that, as $C$ and $JC$ are $\Gamma$-invariant, being lifts of vector fields from $M$, then the $\Gamma$-equivariance of $f:=\omega(C,JC)$ is exactly the equivariance of $\omega$. Thus it follows that $\theta'$ has the same $\Gamma$-equivariance as $\theta$, and so the two one-forms are cohomologuous. Hence the form
\begin{equation}
\Omega':=g^{-1}dd^cg
\end{equation}
descends to $M$ to an LCK metric with positive potential with Lee form $\theta'\in\tau$, and the proof is finished.
\end{proof}

\begin{remark}
Let us note that the above construction of an LCK metric with potential is natural and only depends on $\Omega$ and on $C$. In particular, if $\Omega$ is already $JC$-invariant, which will imply that the metric is Vaisman, then we have $f_t=f$ and the solution of $\eqref{omega4}$ is then $g_t=(1-\cos t)f$, so in particular the potential $g=f$ remains unchanged. 
\end{remark}

\begin{remark}
On the other hand, for a metric $\Omega$ which is not $JC$-invariant, the above construction gives us a countable set of metrics with potential associated to the de Rham class of $\theta$. Indeed, we considered the potential $g_{[1]}:=g$, but for any $n\in\NN^*$, the potential $g_{[n]}:=1/2n\pi\int_{0}^{2n\pi}g_tdt$ works as well.
\end{remark}

\section{Existence of Vaisman metrics}\label{eVaisman}

In this section we are interested in giving a proof of \ref{thVa} and of \ref{notLCK}. We start by giving the main proposition, which will directly imply the general criteria.

Let $(M, J, \Omega,\theta)$ be a Vaisman manifold with corresponding fundamental vector fields $B$ and $A=JB$.  Then $A,B\in\aut(M,J,\Omega)$ generate a holomorphic $\RR^2$ action on $M$, and we will denote by $G$ the image of $\RR^2$ in $\Aut_0(M,J,\Omega)$. Since the Lie group $\Aut_0(M,J,\Omega)$ is compact, we can take the closure of $G$ in it, obtaining thus a compact torus $\TT\subset \Aut_0(M, J,\Omega)$. The torus $\TT$ is not purely real, since both $A$ and $B$ are in $\Lt\cap J\Lt$. In fact, we have:

\begin{proposition}\label{purr}
Let $(M,J,[\Omega],[\theta]_{dR})$ be a strict LCK manifold and $\TT\subset \Aut_0(M,J,[\Omega])$ be a compact torus. If $\TT$ is not purely real, then $[\Omega]$ is Vaisman and $\Lt\cap  J\Lt=\RR\{A,B\}$, where $B=-JA$ is the Lee vector field of some Vaisman metric in $[\Omega]$.
\end{proposition}
\begin{proof}
Choose a $\TT$-invariant LCK structure $(\Omega,\theta)$ in the conformal class $[\Omega]$, so that for any $X\in\Lt$, $d(\theta(X))=\LL_X\theta=0$. Let $0\neq C\in \Lt$ with $D:=JC\in\Lt$. Then both $\theta(C)$ and $\theta(D)$ are constant. However, we cannot have $\theta(C)=\theta(D)=0$. Indeed, if it was the case, then: 
\begin{align*}
0=&\iota_{[C,D]}\Omega=\LL_C\iota_D\Omega-\iota_D\LL_C\Omega=\\
=&d\iota_C\iota_D\Omega+\iota_Cd\iota_D\Omega=\\
=&d(-\norm{C}^2)+\theta\iota_C\iota_D\Omega
=d_\theta(-\norm{C}^2)
\end{align*}
implying, by \ref{ldtheta}, that $\norm{C}^2=0$, contradiction.
Hence, if $\theta(C)=a$ and $\theta(D)=b$, then $X:=aD-bC\neq0$ still verifies $X\in\Lt$ and $JX\in\Lt$ and, moreover, $\theta(X)=0$, so $\theta(JX)\neq 0$. Therefore, we can suppose from the beginning that $\theta(C)=1$ and $\theta(D)=0$.

Let $f:=\norm{C}^2_\Omega$, which is an everywhere positive function since $C$ cannot have any zeros. Take $\Omega':=\frac{1}{f}\Omega$, with corresponding Lee form $\theta'=\theta-d\ln f$. Then, since $f$ is preserved by both $C$ and $D$, we still have $\theta'(C)=1$ and $\theta'(D)=0$, and $C,D\in\aut(M,J,\Omega')$. 

Let $\eta:=\iota_C\Omega'$. Then we have:
\begin{equation*}
d\eta=\LL_C\Omega'-\iota_Cd\Omega'=-\theta'(C)\Omega'+\theta'\wedge\eta
\end{equation*} 
or also $\Omega'=d_{\theta'}(-\eta)$. Since $D$ preserves both $C$ and $\Omega'$, it also preserves $\eta$. Moreover, we have $1=\norm{C}^2_{\Omega'}=\eta(D)$. Hence we get:
\begin{align*}
0=&\LL_{D}\eta=d\iota_D\eta+\iota_Dd\eta=\iota_D(-\Omega'+\theta'\wedge\eta)=\\
=&-J\eta+\theta'(D)\eta-\theta'\eta(D)=-J\eta-\theta'.
\end{align*}
Finally, this implies that $\eta=J\theta'$, so that $C$ is actually the Lee vector field $B$ of $\Omega'$. Since $C$ is holomorphic and preserves $\Omega'$, it is also Killing, so $2\nabla\theta'=d\theta'=0$, that is, $\Omega'$ is Vaisman.

Finally, since a Vaisman metric is unique in its conformal class up to multiplication by constants, it follows that $\Lt\cap J\Lt=\RR\{C,D\}=\RR\{A,B\}$. 
\end{proof}

Note that this immediately implies \ref{notLCK}.

\begin{proof}[\textbf{Proof of \ref{thVa}}]
As we already noted at the beginning of the section, if $M$ admits a Vaisman metric then the corresponding holomorphic vector fields $B$ and $A=JB$ sit in the Lie algebra of a torus in $\Aut_0(M,J)$.

Conversly, suppose $\TT\subset\Aut_0(M,J)$ is not purely real. Take any LCK metric $(\Omega,\theta)$ and average it over $\TT$, in order to get a $\TT$-invariant LCK metric. Hence we have $\TT\subset \Aut_0(M,J,\Omega)$, and we can apply \ref{purr} in order to get the conclusion.  
\end{proof}

\section{Maximal torus actions}

The main goal of this section is to give a proof of \ref{maxTor}, as a consequence of the previous results, together with our result concerning toric LCK manifolds of \cite{is17}. 

Let $(M,J,[\Omega],[\theta])$ be a compact LCK manifold. There are two natural Lie algebras of vector fields one can consider in this context, which we present next. 

\begin{definition}
A vector field $X\in\Gamma(TM)$ is called \textit{horizontal} for $([\Omega],[\theta])$ if $\LL_X\Omega=\theta(X)\Omega$ for some (and hence any) form $\Omega\in[\Omega]$. We denote by $\aut'(M,[\Omega])$ the set of horizontal vector fields, and note that:
\begin{equation*}
\aut'(M,[\Omega])\subset\aut(M,[\Omega]):=\{X\in\Gamma(TM)|\LL_X\Omega=f_X\Omega, \ \ f_X\in\ce(M)\}.
\end{equation*}
Moreover, $\aut'(M,[\Omega])$ inherits the structure of a Lie subalgebra of $\Gamma(TM)$.
\end{definition}

\begin{remark} If a vector field generates an $\Ss^1$-action and is a horizontal vector field for $([\Omega],[\theta])$, then the $\Ss^1$-action is horizontal with respect to $[\theta]$ in the sense of \ref{defTor}. 
\end{remark}

Inside $\aut'(M,[\Omega])$ there is another natural Lie algebra, namely the one given by twisted Hamiltonian vector fields.

\begin{definition}
A vector field $X\in\Gamma(TM)$ is called \textit{twisted Hamiltonian} for $([\Omega],[\theta])$ if for some (and so any) representative $\Omega\in[\Omega]$, there exists $f_X\in\ce(M)$ so that $\iota_X\Omega=d_\theta f_X$. Twisted Hamiltonian vector fields form a Lie subalgebra of horizontal vector fields, denoted by $\ham(M,[\Omega])$.
\end{definition}

Note that these definitions are conformally invariant. The claims that are made above are easy to check, but for the complete proofs and for a motivation of these definitions, one can consult the paper of Vaisman \cite{v85}, where they were first considered.

In particular, we call an action of a connected Lie group $G$ on $(M,J,[\Omega])$ twisted Hamiltonian if $\Lie(G)\subset\ham(M,[\Omega])$. Moreover, if $n$ is the complex dimension of $M$, the manifold $(M,J,[\Omega])$ together with a torus $\TT^n$ that acts on the manifold effectively by biholomorphisms and in a twisted Hamiltonian way is called a \textit{toric LCK manifold}. These kind of manifolds were studied recently in \cite{p}, \cite{mmp} and \cite{is17}. 

\begin{remark}\label{isotropic}
It was shown in \cite[Proposition~3.9]{mmp}, although not explicitly stated, that the orbits of a horizontal torus action of $\TT^n$ on $(M,[\Omega])$ are isotropic for $\Omega$. The main point is that, if $(\Omega,\theta)$ is $\TT^n$-invariant, then for any $X,Y\in \Lie(\TT^n)\subset\aut(M)$, as $X,Y\subset\ker\theta$, we have:
\begin{align*}
0=&\iota_{[X,Y]}\Omega=\LL_X\iota_Y\Omega-\iota_Y\LL_X\Omega=d\iota_X\iota_Y\Omega+\iota_Xd\iota_Y\Omega=d_\theta(\Omega(Y,X)).
\end{align*}
This implies, by \ref{ldtheta}, that $\Omega(Y,X)=0$.
\end{remark}

It is not difficult to see that if $[\Omega]$ is exact, then horizontal actions of compact tori coincide with twisted Hamiltonian ones, see the above references for details. Also, as shown in \cite[Lemma~3.7]{mmp}, this is also the case if $\hat M$ is simply connected. In fact, when the dimension of the torus is maximal, we do not need any hypothesis for this equivalence to hold. The proof of this follows the lines of the one from \cite{is17}, for this matter we will skip some of the details:

\begin{proposition}\label{tHam}
Let $(M,J,[\Omega])$ be a compact LCK manifold of complex dimension $n$ and let $\TT^n$ be a torus that acts effectively by biholomorphisms on the manifold. Then the action is twisted Hamiltonian if and only if it is horizontal.
\end{proposition}
\begin{proof}
Clearly, we only need to show the if direction, so let us suppose that the action is horizontal. Let us fix $\Omega\in[\Omega]$ which is $\TT^n$ invariant, with Lee form $\theta$, so that $\Lt:=\Lie(\TT^n)\subset\ker\theta$. Here and in all that follows, we identify $\Lt$ with a Lie subalgebra of $\aut(M,J)$. Also let $(\hat M,J,\Omega_K)$ be the corresponding minimal \Ka\ cover. Then, by \ref{hor}, we have a lifted action by biholomorphisms of $\TT^n$ on $(\hat M,J)$, and as it is not difficult to see, this action is also symplectic with respect to $\Omega_K$.

By the principal orbit theorem, see for instance \cite{br}, there exists a dense connected open subset $M_0\subset M$ on which $\TT^n$ acts locally freely. Moreover, as $\TT^n$ is abelian and acts effectively on $M$, it acts in fact freely on $M_0$. The preimage $\hat M_0$ of $M_0$ in $\hat M$ is exactly the dense connected open subset of $\hat M$ on which $\TT^n$ acts freely. 

Now the proof of \ref{purr} shows that $\Lt\subset\ker \theta$ implies that $\Lt\cap J\Lt=\{0\}$. Thus we have a complex linear injection $\Lt\oplus J\Lt\rightarrow \aut(M,J)$ generating a holomorphic action of $\TT^c:=(\CC^*)^n$ on $M$ and on $\hat M$. As for any $\xi\in\Lt$, $\xi$ has no zeroes on $M_0$, and as at any point of $M_0$, $\Lt$ is  orthogonal to $J\Lt$ with respect to the metric $g:=\Omega(\cdot,J\cdot)$ by \ref{isotropic}, it follows that the action of $\TT^c$ is locally free on $M_0$, and so also on $\hat M_0$. 

Let us fix $x\in\hat M_0$ and let $H=\{g\in \TT^c| g.x=x\}$, which by the above discussion is a closed discrete subgroup of $\TT^c$. We have a holomorphic embedding $F:\TT^c/H\rightarrow \hat M_0$, $g\mapsto g.x$. As $\dim_{\CC}\TT^c/H=\dim_\CC\hat M_0$, $F$ must be an open embedding, and as $\hat M_0$ is connected, $F$ is thus a biholomorphism. In particular, as $\hat M_0$ is dense in $\hat M$, $H$ acts trivially on the whole of $\hat M$, so we have a well-defined effective action of $\TT^c/H$ on $\hat M$. 

Now, as $\Gamma$ commutes with $\TT^c$ and preserves $\hat M_0$, it follows easily that $\Gamma\subset\TT^c/H$. Thus, given $\id_{\hat M}\neq \gamma\in\Gamma$, there exists a subgroup $\RR\cong G\subset \TT^c/H$ containing $<\gamma>$ as a subgroup. $G$ acts by biholomorphisms on $\hat M$ and this action clearly commutes with $\Gamma$, so descends to an effective $\Ss^1$ action on $M$. By definition, this action is vertical with respect to $[\theta]$. One can average $(\Omega,\theta)$ over this $\Ss^1$-action to obtain an exact $\TT^n$-invariant LCK metric, which is in particular toric for the given action of $\TT^n$. This is exactly the construction of \cite[Lemma~5.1]{is17}, where the details can be found.

Finally, by applying the result of \cite{is17}, it follows that there exists a Vaisman structure on $(M,J)$ with Lee class $[\theta]$. But by \cite{llmp} any $d_\theta$-closed form on $M$ is $d_\theta$-exact. Therefore, any LCK form on $(M,J)$ is exact, and so the torus action for the inital LCK form $[\Omega]$ was twisted Hamiltonian. 
\end{proof}

Now, summing up, we get:
\begin{proof}[\textbf{Proof of \ref{maxTor}}]
Let $\TT^n\subset\Aut(M,J)$, let $\tau\in\LL(M,J)$ and let $(\Omega,\theta)$ be a $\TT^n$-invariant LCK structure with $\theta\in\tau$. If $\Lt\subset\ker\theta$, then the above result implies that $(M,J)$ admits a Vaisman structure $(\Omega',\theta')$ with $\theta'\in\tau$.

If not, then identify $\TT^n$ with $(\Ss^1)^n$ and let $\xi_1,\ldots, \xi_n$ be the fundamental vector fields generating each of the $\Ss^1$-actions on $M$. As $\theta$ does not vanish on the whole of $\Lt$, there exists at least one $\xi=\xi_k$, $k\in\{1,\ldots,n\}$ generating a vertical $\Ss^1$-action with respect to $[\theta]$. Thus, by applying \ref{crPotential}, it follows that $\theta$ is the Lee form of an LCK metric with positive potential. 
\end{proof}

In what follows, we present two examples which show that \ref{tHam} is false if we forget any of the two hypotheses: either that the torus has maximal dimension, or that the action is horizontal. 

\begin{example}({\bf Non-diagonal primary Hopf surfaces})
These are complex LCK surfaces which admit a holomorphic effective action of $\TT^2$, but which are not toric. A non-diagonal primary Hopf surface $M$ is a quotient of $\CC^2-\{0\}$ by a group $\Gamma\cong\ZZ$ generated by:
\begin{equation*}
\gamma(z_1,z_2):=(\be z_1, \be^m z_2+\la z_1^m)
\end{equation*} 
with $\be,\la\in \CC^*, |\be|<1$ and $m\in\NN^*$. These manifolds are known to admit LCK metrics, cf. \cite{go}. We will show that $\Aut(M,J)$ contains a unique maximal torus $\TT^2$ which is purely real. Thus $(M,J)$ does not admit a Vaisman metric, and the action of $\TT^2$ on $(M,J)$ is not toric with respect to any LCK metric on $(M,J)$. These two facts had already been shown in \cite{b00} and in \cite{mmp} respectively. Moreover, $M$ does indeed admit an LCK metric with positive potential, confirming \ref{maxTor}.

The complex Lie algebra $\Lg:=\aut(M,J)=\aut(\CC^2-\{0\}, J)^\Gamma$ was determined in \cite{mmp}, and is given by $\Lg=\CC\{Z_1=z_1\delz{1}+mz_2\delz{2}, Z_2=z_1^m\delz{2}\}$, which is commutative. The complex flow of a vector field $W=aZ_1+bZ_2\in\Lg$ is given by:
\begin{equation*}
\Phi^u_W(z_1,z_2)=(\e^{au}z_1,\e^{amu}(z_2+buz^m_1)), u\in\CC.
\end{equation*}

Hence $\Phi^1_W=\id$ gives $a\in 2\pi i\ZZ$ and $b=0$, while $\Phi^1_W=\gamma$ gives $a\in c+2\pi i\ZZ$ and $b=\frac{\la}{\be^m}$, where $c\in\CC$ is chosen so that $\e^c=\be$. Thus we obtain two real vector fields with closed orbits $\xi_1=\Re(2\pi i Z_1)$ and $\xi_2=\Re(c Z_1+\frac{\la}{\be^m}Z_2)$, generating a maximal torus $\TT^2\subset \Aut(M,J)$. Moreover, $\TT$ is indeed purely real. 
\end{example}

\begin{example}({\bf Inoue-Bombieri surfaces of type $S^+$}). 

These are  compact complex surfaces $S_t^+$ introduced by Inoue in \cite{in}, indexed by $t\in\CC$, which admit an LCK metric if and only if $t\in\RR$, by \cite{tr} and \cite{b00}. We will show that for $t\in\RR$, the LCK structure $(\Omega,\theta)$ constructed in \cite{tr} admits a $[\theta]$-horizontal holomorphic $\Ss^1$-action which is not twisted Hamiltonian, showing thus that the dimension hypothesis on the torus in \ref{tHam} is a necessary condition.

The manifolds $S_t^+$ are obtained as quotients of $\HH\times \CC$ by discrete groups of affine transformations as follows. Let $N=(n_{ij})\in SL(2,\ZZ)$ be a matrix with real eigenvalues $\al$ and $1/\al$ ($\al>1$) and with real corresponding eigenvectors $(a_1,a_2)$ and $(b_1,b_2)$. Fix $p,q\in \ZZ$, $r\in\ZZ^*$ and $t\in\CC$, and let $(c_1,c_2)\in\RR^2$ be defined by: 
\begin{equation*}
(c_1,c_2)(\id-N^t)=(e_1,e_2)+\frac{b_1a_2-b_2a_1}{r}(p,q)
\end{equation*}
where for $i=1,2$, $e_i$ is given by:
\begin{equation*}
e_i=\frac{1}{2}n_{i1}(n_{i1}-1)a_1b_1+\frac{1}{2}n_{i2}(n_{i2}-1)a_2b_2+n_{i1}n_{i2}b_1a_2.
\end{equation*}
Let $G^+\subset\Aut(\HH\times \CC)$ be the group generated by:
\begin{align*}
g_0(w,z)&=(\al w, z+t)\\
g_i(w,z)&=(w+a_i,z+b_iw+c_i) , \text{  } i=1,2\\
g_3(w,z)&=(w,z+\frac{b_1a_2-b_2a_1}{r})
\end{align*}
and let $S^+_t:=\HH\times\CC/G^+$. Here we have denoted by $w$ the holomorphic coordinate on $\HH$ and by $z$ the holomorphic coordinate on $\CC$. Then $H^1(S^+_t,\RR)$ is of rank one, spanned by the class of $\theta:=\frac{d \Im w}{\Im w}$. 

As it was shown in \cite{in}, the only holomorphic vector field on $S^+_t$ is $Z=\delz{}$. Its complex flow is given by $\Phi^\la_Z(w,z)=(w,z+\la)$, and the only solution $\la\in\CC$ to $\Phi^1_{\la Z}\in G^+$ is $\la_0=\frac{b_1a_2-b_2a_1}{r}\in\RR$, in which case $\Phi^1_{\la_0 Z}=g_3$. Hence there exists only one torus, which is one-dimensional, $S^1\subset\Aut_0(S^+,J)$, generated by $\xi=\Re\la_0 Z$, and this torus is $[\theta]$-horizontal by \ref{hor} since $\theta(\xi)=0$. 

Now for $t\in\RR$, an LCK metric on $S^+_t$ with Lee form $\theta$ is given by:
\begin{equation*}
\Omega=i\left(\frac{1+(\Im z)^2}{(\Im w)^2}dw\wedge d\bar w -\frac{\Im z}{\Im w}(dw\wedge d\bar z+dz\wedge d\bar w)+dz\wedge d\bar z\right )
\end{equation*}
and we have $\iota_\xi\Omega=\la_0(d\Im z-\frac{\Im z}{\Im w}d\Im w)=\la_0d_\theta \Im z$. The relation that we wrote is valid on $\HH\times\CC$, however $\Im z$ does not define a function on $S^+_t$. In fact, the form $\be:=d_\theta\Im z$ is not $d_\theta$-exact on $S^+_t$. Indeed, if it were, then there would exist a $G^+$-invariant function $g\in\ce(\HH\times\CC)$ verifying $d_\theta(g-\Im z)=0$ on $\HH\times \CC$. If we let $h:=g-\Im z\in\ce(\HH\times\CC)$, then $h$ would verify $\Im w dh=d\Im w$, and so $h(w,z)=C\Im w$ for some $C\in\RR$. But then the function $g=\Im z+C\Im w$ is not $g_0$-invariant, which is a contradiction. Hence the vector field $\xi$ is not twisted Hamiltonian for $\Omega$. 

\end{example}

\section{Holomorphic torus principal bundles}\label{torusPrincipal}

Let  $\mathbf{T}=\Lt/\Lambda$ be a compact complex torus of dimension $n$, let $N$ be a compact complex manifold and let $\pi:M\rightarrow N$ be a holomorphic $\mathbf{T}$-principal bundle over $N$. Its Chern class is an element: 
\begin{equation*}
c^\ZZ(\pi)\in H^2(N,\Lambda)\cong H^2(N,\ZZ)\otimes\Lambda.
\end{equation*}
The inclusion $\Lambda\subset \Lt$ induces a natural map $H^2(N,\Lambda)\rightarrow H^2(N,\Lt)\cong H^2(N,\CC)\otimes \Lt$, and we will denote by $c(\pi)$ the image of $c^\ZZ(\pi)$ under this map.
The class $c(\pi)$ has a well defined rank. If we choose $\CC$-basis for both $\Lt$ and $H^2(N,\CC)$, then $c(\pi)$ can be represented by a $2n\times b_2(N)$ matrix over $\CC$, and then the rank of $c(\pi)$ is the rank of this matrix. 

Note that if the rank of $c(\pi)$ is 1, then there exists a minimal element $a\in\Lambda$, unique modulo sign, such that the non-torsion part of $c^\ZZ(\pi)$ writes $c^\ZZ(\pi)_0=c_1^\ZZ(\pi)\otimes a$ with $c_1^\ZZ(\pi)\in H^2(N,\ZZ)$. If $c_1(\pi)$ is the image of $c_1^\ZZ(\pi)$ under $H^2(N,\ZZ)\rightarrow H^2(N,\CC)$, then we will have $c(\pi)=c_1(\pi)\otimes a$, and again $c_1(\pi)$ is uniquely defined modulo sign. So it makes sense to ask weather $c_1(\pi)$ is a positive or negative class, i.e. weather $c_1(\pi)$ or $-c_1(\pi)$ can be represented by a \Ka\ form on $N$. In the affirmative case, we will call the class $c(\pi)$ \textit{definite}.     

By a theorem of Blanchard \cite{bl}, when $N$ is of \Ka\ type, $M$ carries a \Ka\ metric if and only if the rank of $c(\pi)$ is 0. On the other hand, a theorem of Vuletescu \cite{vu} states that if $n=1$ and the rank of $c(\pi)$ is 2, then $M$ cannot admit LCK metrics. 

As a direct application of our existence criterion for Vaisman metrics and of \ref{notLCK}, we obtain a characterisation of manifolds of $LCK$-type among all the compact torus principal bundles over compact complex manifolds.   

\begin{proposition}\label{principal}
Let $\mathbf{T}$ be a complex compact $n$-dimensional torus and $\pi:M\rightarrow N$ be a $\mathbf{T}$-principal bundle over a compact complex manifold $N$. Then $M$ admits a (strict) LCK metric if and only if $n=1$ and the Chern class of $\pi$ is of rank $1$ and definite. In this case, $M$ is of Vaisman type.
\end{proposition}
\begin{proof}
Suppose that $M$ admits a strict LCK metric. The complex torus $\mathbf{T}$ acts holomorphically and effectively on $M$, so, by \ref{thVa}, $M$ admits a Vaisman metric $(\Omega,\theta)$. Let $B$ be the Lee vector field with $\theta(B)=1$ and $A:=JB$. By  \ref{purr}, $n=1$ and $\Lt=\Lie(\mathbf{T})$ is spanned by $A$ and $B$. Here, we identify $\Lt$ with its isomorphic image as a subalgebra of $\Gamma(TM)$.

Since the $\mathbf{T}$-invariant $1$-forms $\theta_1=J\theta$ and $\theta_2=\theta$ verify $\theta_i(X_j)=\delta_{ij}$, for $i,j=1,2$, where $X_1=A$ and $X_2=B$, there will exist some linear combination of them giving a connection form $\al\in\ce(T^*M\otimes \Lt)$ in $\pi$. More precisely, if we denote by $\xi_1$, $\xi_2$ the fundamental vector fields of the action, and let $G=(g_{ij})$ be the matrix of $\{X_1, X_2\}$ in the basis $\{\xi_1$, $\xi_2\}$ of $\Lt$, then the connection form will be given by:
\begin{equation*}
\al:=(g_{11}\theta_1+g_{21}\theta_2)\otimes\xi_1+(g_{12}\theta_1+g_{22}\theta_2)\otimes\xi_2.
\end{equation*}
Indeed, it is $\mathbf{T}$-invariant and we have $\al(\xi_i)=\xi_i$ for $i=1,2$. Moreover, since $d\theta=0$, its curvature is: 
\begin{equation*}
\Theta:=d\al=d\theta_1\otimes g_{11}\xi_1+d\theta_1\otimes g_{12}\xi_2=dJ\theta\otimes A.
\end{equation*}
It is a basic form, so given by $\Theta=\pi^*\eta\otimes A$, with $\eta\in\Omega^2(N)$, and $\eta\otimes A$ represents the Chern class $c(\pi)\in H^2(N,\Lt)$. Then clearly $c(\pi)$ is of rank 1, and moreover, it is definite since the form $-\eta$ is a \Ka\ form on $N$. The last assertion comes from the fact that, as $\Omega$ is Vaisman, we have $-dJ\theta=\Omega-\theta\wedge J\theta$, so the $(1,1)$-form $-dJ\theta$ is strictly positive on $Q:=\ker \theta\cap \ker J\theta\subset TM$. But $Q$ is exactly the horizontal distribution given by the connection $\al$, and so identifies with $TN$ via $\pi_*$.

The converse statement is a well known result, we send the reader to \cite{t99} for the detailed construction.
\end{proof}

\section{Analytic irreducibility of complex manifolds of LCK type}\label{anIrred}

It is not very difficult to see that a product metric cannot be LCK (\cite{v80}), but whether an LCK manifold must be analytically irreducible is still an open question.  Under additional hypotheses, the answer is known to be positive (\cite{t99}, \cite{opv}). In this section we wish to enlarge the list of hypotheses implying the analytic irreducibility of the manifold. 

One of the results in this direction  is due to Tsukada \cite{t99}, which we can also obtain as a direct consequence of \ref{thVa}:

\begin{proposition}(\cite{t99})\label{iredva}
Let $M_1$ and $M_2$ be two compact complex manifolds of Vaisman type. Then $M:=M_1\times M_2$ admits no LCK metric.
\end{proposition}
\begin{proof}
By \ref{thVa}, the groups of biholomorphisms $\Aut(M_i)$ contain tori $\TT_i$ which are not purely real, for $i=1,2$. Then the Lie algebra $\Lt$ of the torus $\TT:=\TT_1\times \TT_2\subset \Aut(M)$ verifies $\dim_\CC\Lt\cap J\Lt=2$. Hence, by \ref{notLCK}, $M$ cannot admit an LCK metric. 
\end{proof}

Tsukada obtained \ref{iredva} as a corollary to the following result: 

\begin{theorem}(\cite{t99})
Let $(M,\Omega)$ be a compact Vaisman manifold and let $\mathcal{F}$ be the canonical foliation on $M$ generated by the Lee and the anti-Lee vector fields. Then $\mathcal{F}$ has a compact leaf.
\end{theorem}

We can further exploit this and obtain the following, more general, result:

\begin{theorem}\label{ired}
Let $M_1$, $M_2$ be two compact complex manifolds and suppose that $M_1$ is of Vaisman type. Then $M:=M_1\times M_2$ admits no LCK metric. 
\end{theorem}

\begin{proof}
Suppose $M$ admits some LCK metric. Then, for any $x\in M_1$, this metric restricted to $\{x\}\times M_2\cong M_2$ gives an LCK metric on $M_2$.

Since $M_1$ is of Vaisman type, there exists $\TT_1\subset\Aut(M_1)$ whose Lie algebra $\Lt_1$ verifies $\dim_\CC \Lt_1\cap J\Lt_1=1$. The induced torus $\TT=\TT_1\times \{\id_{M_2}\}\subset\Aut(M)$ is still not purely real, so by  \ref{thVa},  $M$ is of Vaisman type and $\Lt:=\Lie(\TT)$ contains the corresponding Lee vector field $B$.

Let $\Omega$ be a Vaisman metric on $M$ which, after eventual averaging, is $\TT$-invariant. Then for any $y\in M_2$, $\Omega$ restricted to $M_1\times \{y\}\cong M_1$ must be Vaisman. Indeed, by construction, the Lee vector field $B$ is tangent to $M_1$, and \cite[Theorem~5.1]{v80} states that any submanifold of a Vaisman manifold that is tangent to the Lee vector field is again Vaisman with the induced metric. Let now $E\subset M_1$ be a closed leaf of the canonical foliation on the Vaisman manifold $M_1$, as in the above theorem. Clearly, after choosing $\mathit{O}\in E$, $E$ has the structure of an elliptic curve whose tangent bundle is generated by $B$ and $JB$ restricted to $E$. Hence, the submanifold $i:Y=E\times M_2 \rightarrow M$ together with $i^*\Omega$ is Vaisman. At the same time, $Y\rightarrow M_2$ is a trivial $E$-principal bundle, so we arrive at a contradiction via \ref{principal}.

\end{proof}

Also, using the result which states that a submanifold of a Vaisman manifold must contain the leaves of the canonical foliation, one has:

\begin{proposition}\label{corired}
A compact complex manifold of Vaisman type is holomorphically irreducible. 
\end{proposition}

\begin{proof}
Let $M=M_1\times M_2$ be the compact complex manifold with the product complex structure, and suppose it admits a Vaisman metric $\Omega$ with corresponding canonical foliation $\mathcal{F}$ generated by $B, JB$. Then, by \cite[Thm~3.2]{t97}, for any $(x_1,x_2)\in M$, both the submanifolds $M_1\times\{x_2\}$ and $\{x_1\}\times M_2$ of $M$ contain the leaves of $\mathcal{F}$, which is impossible.
\end{proof}

On the other extreme, we have the following result, also obtained in \cite{opv} in a different manner:
\begin{theorem}\label{produs}
Let $M_1$, $M_2$ be two connected complex manifolds, and suppose that $M_1$ is compact and verifies the $\partial\cp$-lemma. Moreover, if $\dim_\CC M_1=1$, then suppose its genus $g$ is $1$ and $M_2$ is compact, or that $g=0$. Then $M:=M_1\times M_2$ admits no (strict) LCK metric.
\end{theorem}

\begin{proof}
Suppose $M$ admits an LCK form $\Omega$ with corresponding Lee form $\theta$. Denote by $p_i:M\rightarrow M_i$, $i=1,2$ the canonical projections. We have, by the K\"unneth formula, an isomorphism $p_1^*\oplus p_2^*:H^1(M_1,\RR)\oplus H^1(M_2,\RR)\rightarrow H^1(M,\RR)$, meaning that there exist two closed forms $\theta_i\in\ce(T^*M_i)$, $i=1,2$, such that $\theta$ is cohomologuous to $p_1^*\theta_1+p_2^*\theta_2$. After a conformal change of $\Omega$, we can suppose that $\theta=p_1^*\theta_1+p_2^*\theta_2$. 

Since an LCK metric on $M$ induces one on $M_1$, by a result of Vaisman \cite{v80}, the induced metric is globally conformal to a \Ka\ metric. Suppose moreover that $n:=\dim_\CC M_1>1$. Then this implies that $\theta_1$ is exact on $M_1$. Again, by a global conformal change of $\Omega$, we can suppose that $\theta_1=0$.
Then the conclusion follows from  \ref{thetaz} bellow. 

If $\dim_\CC M_1=1$, then the induced LCK form on $M_1$ is automatically \Ka, so we know nothing of $\theta_1$. However, if $g=1$, then $M\rightarrow M_2$ is a trivial principal elliptic bundle, so by \ref{principal}, $M$ cannot admit any strict LCK metric. If $g=0$ then $M_1$ is simply connected, so we can again apply \ref{thetaz}.
\end{proof}

\begin{lemma}\label{thetaz}
Let $M_1$ and $M_2$ be two connected complex manifolds, with $M_1$ compact. Suppose that $(\Omega,\theta)$ is an LCK form on the manifold $M=M_1\times M_2$ such that $i^*[\theta]=0$ in $H^1(M_1,\RR)$, where $i:M_1\rightarrow M$ is the inclusion $x\mapsto (x,y)$ for some $y\in M_2$. Then $M$ (and thus also both $M_1$ and $M_2$) are of \Ka\ type.
\end{lemma}
\begin{proof} 
As before, after an eventual conformal change of $\Omega$, we have $\theta=p_1^*\theta_1+p_2^*\theta_2$ with $\theta_i\in\ce(T^*M_i), i=1,2$. By hypotheses, $\theta_1=df$, with $f\in\ce(M_1)$, so again, by replacing $\Omega$ with $\e^{-p_1^*f}\Omega$, we can suppose that $\theta=p_2^*\theta_2$.  

The algebra of differential forms on $M$, $\ce(\bigwedge T^*M)$, has two compatible gradings: one given by the degree of the forms, and the second one induced by the splitting $T^*M=p^*_1T^*M_1\oplus p_2^*T^*M_2$. With respect to this second splitting, write the differential $d=d_1+d_2$, and write $\Omega=\Omega_1+\Omega_{12}+\Omega_2\in\ce(\bigwedge^2 T^*M)$, where: 
\begin{equation*}
\textstyle\bigwedge^2 T^*M=\textstyle\bigwedge^2 p_1^*T^*M_1\oplus p_1^* T^*M_1\otimes p_2^* T^*M_2\oplus \textstyle\bigwedge^2 p_2^* T^*M_2.
\end{equation*}
Then the equation $d\Omega=p^*_2\theta_2 \wedge \Omega$ gives, in the homogeneous parts $\bigwedge^2 p_1^*T^*M_1\otimes\bigwedge^1 p_2^*T^* M_2$ and $\bigwedge^3 p_1^* T^*M_1$:
\begin{equation}\label{we1}
\begin{split}
d_2\Omega_1+d_1\Omega_{12}&=p_2^*\theta_2\wedge \Omega_1\\
d_1\Omega_1&=0.
\end{split}
\end{equation} 
Let $n$ be the complex dimension of $M_1$. If we take the wedge of the first equation in \eqref{we1} with $\Omega_1^{n-1}$, we obtain:

\begin{equation}\label{we2}
\frac{1}{n}d_2(\Omega_1^n)+d_1(\Omega_{12}\wedge\Omega_1^{n-1})=p_2^*\theta_2\wedge\Omega_1^n.
\end{equation}

On the other hand, the compactness of $M_1$ implies that $p_2$ is a proper submersion, so it induces a push forward map on forms given by fiberwise integration:
\begin{align*}
(p_2)_*:\ce(\textstyle\bigwedge^{2n}p_1^*T^*M_1\otimes \textstyle\bigwedge^kp_2^*T^*M_2)&\rightarrow \ce( \textstyle\bigwedge^kT^*M_2)\\
((p_2)_*\al)_y&:= \int_{M_1\times\{y\}}\al,\ \  y\in M_2.
\end{align*}

Applying this map to \eqref{we2} and using Stokes' theorem, we obtain the following relation on $M_2$:

\begin{equation*}
 dg=n\theta_2g, \ \text{ where } \ g:=(p_2)_*\Omega_1^n\in\ce(M_2). 
\end{equation*}

This relation also reads $d_{n\theta_2}g=0$, with $g\neq 0$ since it is in fact everywhere positive. By \ref{ldtheta}  we obtain that $n\theta_2$ is exact, so also $\theta$ is and $\Omega$ is globally conformal to a \Ka\ metric.
\end{proof}

\begin{remark}
In \cite{opv}, the authors claim a proof of \ref{produs} also for the case when $M_1$ is a Riemann surface of genus $\geq 2$, but we believe that their argument does not hold. However, we are only able to find restrictions on the manifold $M_2$ under the hypothesis that $M_1\times M_2$ admits an LCK metric: \end{remark}

\begin{proposition}\label{iredcurve}
Let $M_1$ be a compact complex curve, let $M_2$ be a complex manifold and suppose that $M:=M_1\times M_2$ admits an LCK metric. Then $M_2$ admits an LCK metric with positive potential. 
\end{proposition}

\begin{proof}
We keep the same notations as before. If $(\Omega,\theta)$ is the LCK form on $M$, we can suppose that $\theta=\theta_1+\theta_2$ with each $\theta_i$ being the pullback of a closed one form from $M_i$, $i=1,2$. Moreover, up to a conformal change of $\Omega$, as $M_1$ is \Ka ian, we can choose $\theta_1$ to be the real part of a holomorphic one form, so that $dJ\theta_1=0$, where $J$ is the product complex structure on $M$.

As before, on the $\bigwedge^2 T^*M_1\otimes \bigwedge^1 T^*M_2$-part, $d\Omega=\theta\wedge\Omega$ gives:
\begin{equation}\label{jeq1}
d_1\Omega_{12}+d_2\Omega_1=\theta_2\wedge\Omega_1+\theta_1\wedge\Omega_{12}.
\end{equation}

Extend $J$ as a derivation acting on forms, and let $d^c=i(\cp-\partial)$. Then, on $M$ we have the commutation relation: 
\begin{equation}\label{com}
[J,d]=d^c.
\end{equation}

The formula $Jd\Omega=J(\theta\wedge\Omega)$, together with $J\Omega=0$ and \eqref{com} gives, on the $\bigwedge^1 T^*M_1\otimes \bigwedge^2 T^*M_2$-part:
\begin{equation}\label{jeq2}
d_2^c\Omega_{12}+d_1^c\Omega_2=J\theta_1\wedge\Omega_2+J\theta_2\wedge\Omega_{12}.
\end{equation} 

Now we apply the push forward map $(p_2)_*$ to equation \eqref{jeq1} and Stoke's theorem in order to obtain:
\begin{align*}
(p_2)_*d_2\Omega_1=(p_2)_*(\theta_2\wedge\Omega_1+\theta_1\wedge\Omega_{12}).
\end{align*}
If we denote by $g$ the strictly positive function on $M_2$ given by $(p_2)_*\Omega_1$, this also reads:
\begin{equation}\label{jeq13}
d_{\theta_2}g=(p_2)_*(\theta_1\wedge\Omega_{12}).
\end{equation} 
We apply $d^c$ to this identity and use equation \eqref{jeq2} together with \eqref{jeq13} to get:
\begin{align*}
d^cd_{\theta_2}g&=-(p_2)_*(\theta_1\wedge d_2^c\Omega_{12})\\
&=-(p_2)_*(\theta_1\wedge J\theta_1\wedge\Omega_2)+J\theta_2\wedge d_{\theta_2}g+(p_2)_*(\theta_1\wedge d_1^c\Omega_2).
\end{align*}
Since we chose $\theta_1$ so that $dJ\theta_1=0$, equation \eqref{com} implies that $d_1^c\theta_1=0$, hence the above simply gives:
\begin{equation*}
d^cd_{\theta_2}g-J\theta_2\wedge d_{\theta_2}g=-(p_2)_*(\Omega_2\wedge \theta_1\wedge J\theta_1).
\end{equation*}

Note that $\al:=\Omega_2\wedge\theta_1\wedge J\theta_1$ is a semipositive 
$(2,2)$-form on $M$ which is strictly positive on a non-empty open subset of $M$ of the form $U\times M_2$,  where $U\subset M_1$ is the open set where $\theta_1$ does not vanish. Then $\eta:=(p_2)_*\al$ is a strictly positive $(1,1)$-form on $M_2$ verifying:
\begin{equation}\label{jeq14}
\eta=d_{\theta_2}d_{\theta_2}^cg.
\end{equation}
Finally, this implies that $\eta=-dJ\omega+\omega\wedge J\omega$, where $\omega:=\theta_2-d\ln g$, so $(\eta,\omega)$ is an LCK metric with positive potential on $M_2$. \end{proof}

\end{document}